\documentclass{amsart}

\usepackage{hyperref}
\usepackage[hyphenbreaks]{breakurl}

\newtheorem{theorem}{Theorem}[section]
\newtheorem{lemma}[theorem]{Lemma}

\theoremstyle{definition}
\newtheorem{definition}[theorem]{Definition}
\newtheorem{example}[theorem]{Example}

\theoremstyle{remark}
\newtheorem{remark}[theorem]{Remark}

\numberwithin{equation}{section}

\def\K{\mathop{\mbox{\bf\Large K}}}
\def\k{\mathop{\mbox{\bf\large K}}}

\begin{document}

\title{The MC algorithm and its applications}

\author{Xiaodong Cao}
\address{Department of Mathematics and Physics,
Beijing Institute of Petro-Chemical Technology,
Beijing, 102617, P. R. China}
\curraddr{}
\email{caoxiaodong@bipt.edu.cn}
\thanks{}

\author{Shuang Chen}
\address{Department of Mathematics and Physics,
Beijing Institute of Petro-Chemical Technology,
Beijing, 102617, P. R. China}
\email{shchen@bipt.edu.cn}
\thanks{}


\date{}

\dedicatory{}


\begin{abstract}
By applying the MC algorithm and the Bauer-Muir transformation for continued fractions, in this paper we shall give six examples to show how to establish an infinite set of continued fraction formulas for certain Ramanujan-type series, such as Catalan's constant, the exponential function, etc.
\end{abstract}

\maketitle

\section{Introduction}
In a celebrated paper of Bailey-Borwein-Plouffe~\cite{BBP}, they proposed the following fast series
\begin{equation}
\pi=\sum_{m=0}^{\infty}
\frac{1}{16^m}\left(\frac{4}{8m+1}-\frac{2}{8m+4}-\frac{1}{8m+5}
-\frac{1}{8m+6}\right).\label{pi-BBP}
\end{equation}
Since this discovery in 1997, many BBP formulas for various mathematical constants have been discovered~(\emph{e.g.}~\cite{BBBW,BBMW,MSW,Zudilin,Lerch,CFormula} and references quoted therein) with the general form
\begin{equation}
\alpha=\sum_{m=0}^{\infty}\frac{1}{b^m}\frac{p(m)}{q(m)},\label{BBP}
\end{equation}
where $\alpha$ is a constant,  $p$ and $q$ are polynomials with integer coefficients, and $b\ge 2$ is an integer numerical base. In the sequel, we shall call the series of right-hand side of \eqref{BBP} a ``BBP-type series", or ``BBP series" for short. 

Assume that a specific Ramanujan-type series~(defined in Sec.~2 below) converges to $\alpha$, it is natural to ask how we construct a new sequence such that it converges to $\alpha$ with increasingly higher speed. Through analysis of the related results, many researchers have observed that this question has an intriguing connection with the first order linear difference equation with the form of $X_m-a(m)X_{m+1}=b(m)$, where both $a(m)$ and $b(m)$ are rational functions in $m$. For the related theory of the first order linear difference equation, the reader is referred to~\cite{Solving DF}. 
About the 1950s, Oscar Perron even showed some examples to explain how to use the Bauer-Muir transformation to prove that a special continued fraction satisfies a difference equation with constant coefficient $a(m)\equiv 1$~(see~\cite[p.~25-37]{Perron}). 
Mortici~\cite{Mor,Mor-1} introduced some elegant ideas of finding the finite continued fraction approximation solution, yielding a lot of interesting consequences. Moreover, Cao and Mortici~\cite{CM} also provided several examples with $a(m)\equiv 1$ by a continued fraction formula of Ramanujan involving  products and quotients of gamma functions~(see \cite[p.~156, Entry 35]{Berndt}). Based on Mortici~\cite{Mor-1} and Gosper's algorithm~(see~\cite{PWZ}), in~\cite{CY} we propose the so-called the MC algorithm~(see Sec.~3 below), which is a constructive iteration algorithm. Occasionally, the MC algorithm may be used to find a simple closed form or guess at continued fraction special solution~(see~\cite[p.~204, Example 1-3]{CY}). However, authors have not proved any new continued fraction formula for the case $a(m)\neq 1$. Following Perron's approach in~\cite{Perron}, in the present paper we shall give six examples to show how to prove that a special continued fraction satisfies a difference equation with $a(m)\neq 1$. Our main ingredient is the employment of a nonlinear modifying factors in its Bauer-Muir transformation for the continued fraction of Catalan's constant or the companion of Catalan's constant. Additionally, for the exponential function, we shall use the simplest constant modifying factors in its Bauer-Muir transformation.

The paper is organized as follows. In Sec.~2, we collect several notations and definitions for later use. In Sec.~3, based on the work in~\cite{CY}, we shall further develop the multiple-correction method, or ``MC algorithm" for short. In Sec. 4 we shall prepare some preliminary lemmas for the theory of continued fractions. From Sec.~5 to 10, we shall sequentially investigate the continued fraction representation for Catalan's constant, a companion of Catalan's constant, the arcsine function, a special Lerch transcendent, a companion of Landau's constants and the exponential function. In the last section, we shall analyze the related perspective of research in this direction.
\section{Notation and definitions}
To state our results clearly, let us introduce some notations and definitions. We use $\mathbb{N}_0$ to denote a set of non-negative integers. The notation $P_l(n)$ or $Q_l(n)$ means a polynomial of degree $l$ in $n$, while $\Phi(l;n)$ denotes a monic polynomial of degree $l$ in $n$. For convenience, the empty product $\prod_{j=0}^{n}d_j$ and sum $\sum_{j=0}^{n}d_j$  for $n <0$ are stipulated to be one and  zero, respectively. Let $\left\{a_n\right\}_{n=1}^\infty$ and $\left\{b_n\right\}_{n=0}^\infty$ be  two sequences of complex
numbers. The generalized continued fraction
\begin{equation*}
\tau=b_0+\frac{a_1}{b_1+\frac{a_2}{b_2+\ddots}}=b_0+
\begin{array}{ccccc}
a_1 && a_2 &       \\
\cline{1-1}\cline{3-3}\cline{5-5}
 b_1 &\!+\!& b_2 &\!+\cdots
\end{array}
=b_0+\K_{n=1}^{\infty}
\left(\frac{a_n}{b_n}\right)
\end{equation*}
is defined as the limit of the $n$th approximant
\begin{equation}
\frac{A_n}{B_n}=b_0+\K_{k=1}^{n}\left(\frac{a_k}{b_k}\right)
\end{equation}
as $n$ tends to infinity. The numerators $A_n$ and denominators $B_n$ of the approximants satisfy the three-term recurrence relations
\begin{equation}
A_{n}=b_{n}A_{n-1}+a_{n}A_{n-2},\quad B_{n}=b_{n}B_{n-1}+a_{n}B_{n-2}\label{Recurrence-relations}
\end{equation}
with initial values $A_{-1}=1,~A_0=b_0,~B_{-1}=0$ and
$B_0=1$, see Berndt~\cite[p.~105]{Berndt}. We also employ the space-saving notation 
\begin{align*}
\frac{a_0}{b_0+}\K_{m=1}^{\infty}\left(\frac{a_m}{b_m}\right):=\frac{a_0}{b_0+\K_{m=1}^{\infty}\left(\frac{a_m}{b_m}\right)}.
\end{align*}
For the theory of continued fractions, we refer to the books Cuyt \emph{et al.} \cite{CPV}, Jones and Thron~\cite{JT}, Khovanskii~\cite{Kh}, Lorentzen and Waadeland~\cite{LW}, Perron~\cite{Perron}, Wall~\cite{Wall}, etc. Let us now collect three definitions in the theory of continued fractions.
\begin{definition} We say that two continued fractions are equivalent if they have the same
sequence of classical approximants~(see~\cite[p.~72]{LW}). We write $b_0+\k\left(a_n/b_n\right)\approx d_0+\k\left(c_n/d_n\right)$ to express that
$b_0+\k\left(a_n/b_n\right)$
 and $d_0+\k\left(c_n/d_n\right)$ are equivalent.
\end{definition}
\begin{definition}
The Bauer-Muir transform of a continued fraction $b_0+\k(a_n/b_n)$ with respect to a sequence $\left\{r_n\right\}$ from $\mathbb{C}$ is the continued fraction $d_0+\k(c_n/d_n)$ whose canonical numerators $C_n$ and denominators $D_n$ are given by
\begin{align*}
&C_{-1}=1,\quad\quad\quad\quad\quad\quad D_{-1}=0,\\
&C_n=A_n+A_{n-1}r_n,\quad D_n=B_n+B_{n-1}r_n
\end{align*}
for $n\in \mathbb{N}_0$, where $\{A_n\}$ and $\{B_n\}$ are the canonical numerators and denominators of $b_0+\k(a_n/b_n)$. See~\cite[p.~76]{LW}.
\end{definition}
\begin{definition}
We shall call $d_0+\k\left(c_n/d_n\right)$ a contraction of $b_0+\k\left(a_n/b_n\right)$ if its classical approximants $\left\{g_n\right\}$ form a subsequence of the classical  approximants $\left\{f_n\right\}$ of $b_0+\k\left(a_n/b_n\right)$. In particular, we call $d_0+\k\left(c_n/d_n\right)$ a canonical contraction of $b_0+\k\left(a_n/b_n\right)$ if
\begin{align}
C_k=A_{n_k},\quad D_k=B_{n_k}\quad \mbox{for $k\in\mathbb{N}_0$,}
\end{align}
where $C_n, D_n, A_n$ and $B_n$ are canonical numerators and denominators of $d_0+\k\left(c_n/d_n\right)$ and $b_0+\k\left(a_n/b_n\right)$,
respectively. See~\cite[p.~83]{LW}.
\end{definition}
In this paper, we shall investigate a more general class of series than BBP-type series.
\begin{definition}
Let a series $\sum_{m=0}^{\infty}t_m$ be convergent. A function $t_m$ is said to be a Ramanujan-type term, or a ``Ramanujan term" for short if it can
be written in the form
\begin{align}
t_m=R(m)\frac{\prod_{i=1}^{uu}(a_im+c_i)!}
{\prod_{j=1}^{vv}(b_jm+d_i)!}
\frac{1}{q^m},\label{tm-def}
\end{align}
in which $q\neq 0$ is a specific constant, and
\begin{enumerate}
  \item $R(m)$ is a rational function in $m$
  \item the $a_i,c_i, b_j$ and $d_j$ are specific integers with
$a_i>0, b_j>0$, and
  \item the quantities $uu$ and $vv$ are finite, nonnegative, specific integers.
\end{enumerate}
\begin{remark} Clearly, a BBP series is a Ramanujan-type series from their definitions.
Here we also note that in the definition of a Ramanujan-type series, the previous condition $q\in (0,+\infty)$ for a quasi-BBP series~(introduced in ~\cite{CY})  is relaxed to $q\neq 0$.
\end{remark}
\end{definition}
Lastly, we give a definition introduced in the previous paper~\cite{CY}.
\begin{definition}
We shall call the integer $l-m$ the degree of a rational function $R(k)=\frac{P_l(k)}{Q_m(k)}$ in $k$, and write $\deg R(k)=l-m$.
\end{definition}

\section{The MC algorithm}
Let a Ramanujan series  $\sum_{m=0}^{\infty}t_m$ converge to constant $\alpha$. If we use the finite sum $\sum_{m=0}^{n-1}t_m$ to approximate or compute $\alpha$ for some comparatively large positive integer $n$, the error term $E(n)$ equals to $\sum_{m=n}^{\infty}t_m$. To evaluate it more accurately, in general, we need to ``separate" second main term $\mathrm{MT}(n)$ from $E(n)$ such that the new error term $E(n)-\mathrm{MT}(n)$ has a faster rate of convergence than $E(n)$ when $n$ tends to infinity. The idea of the MC algorithm is that we can manipulate it in some cases by looking for the proper structure of $\mathrm{MT}(n)$, where $\mathrm{MT}(n)=\frac{\prod_{i=1}^{uu}(a_in+c_i)!}
{\prod_{j=1}^{vv}(b_jn+d_i)!}
\frac{1}{q^n}\mathrm{MC}_k(n)$, and $\mathrm{MC}_k(n)$ usually is likely to be a sum of a polynomial and a finite continued fraction with the form $\frac{\lambda_0}{\Phi(\kappa_0;n)+}\k_{j=1}^{k}\left(\frac{\lambda_j}{\Phi_j(\kappa_j;n)}\right)$, here $\lambda_j\in \mathbb{R}$. In fact, the MC algorithm is a recursive algorithm, and one of its advantages is that by repeating the correction process we can always accelerate the convergence. More precisely, every non-zero coefficient plays a key role in accelerating the rate of approximation. The MC algorithm consists of the following several steps.
\begin{description}
  \item[Step 1] Simplify the ratio $\frac{t_{m+1}/R_{m+1}}{t_m /R_m}$ to bring the form $\frac{P_r(m)}{Q_s(m)}$, where $P_r, Q_s$ are polynomials.
  \item[Step 2] We begin from $k=0$, and  in turn find  the finite continued fraction approximation solution $\mathrm{MC}_k(m)$ of the following difference equation
\begin{align}
X_m-\frac{P_r(m)}{Q_s(m)}X_{m+1}=R(m),\label{First-DF}
\end{align}
until some suitable $k=k^*$ you want.
  \item[Step 3] Substitute the above $k$-th correction function $\mathrm{MC}_k(m)$ into the left-hand side of~\eqref{First-DF} to find the constant $C_k$ and positive integer $K_0$ such that
\begin{equation}
~~\mathrm{MC}_k(m)\!-\!\frac{P_r(m)}{Q_s(m)}\mathrm{MC}_k(m\!+\!1)\!-\!R(m)
\!+\!\frac{C_k}{m^{K_0}}\!=\!O\left(\frac{1}{m^{K_0+1}}\right).~~
\label{Second-DF}
\end{equation}
In what follows, we call $\mathrm{T}_k(m):=\mathrm{MC}_k(m)-\frac{P_r(m)}{Q_s(m)}\mathrm{MC}_k(m+1)-R(m)$  the \emph{test error function} of the $k$-th correction.
  \item[Step 4] Try to discovery the general term in the $k$th-correction. Fortunately, if answer is yes, then we construct the formal continued fraction solution of the difference
equation~\eqref{First-DF}, then propose a reasonable conjecture. Further, if $E_{k^*}(x)\equiv 0$ for some $k^*$, then we obtain a \emph{simple closed form} of the difference equation~\eqref{First-DF}.
  \item[Step 5] Consider the new \emph{quasi-BBP term} appearing in (Step 3)
\begin{align}
tt_m=\frac{1}{m^{K_0}}\frac{\prod_{i=1}^{uu}(a_im+c_i)!}
{\prod_{j=1}^{vv}(b_jm+d_i)!}\frac{1}{q^m},\label{ttm-def}
\end{align}
then repeat (Step 1) to (Step 3).
  \item[Step 6] Define the $k$-th correction error term $E_k(n)$ as
\begin{align}
E_k(n):=\alpha-\sum_{m=0}^{n-1}t_m
-\frac{\prod_{i=1}^{uu}(a_in+c_i)!}
{\prod_{j=1}^{vv}(b_jn+d_i)!}
\frac{1}{q^n}\mathrm{MC}_k(n).
\end{align}

Try to find the rate of convergence of the $k$-th correction error term $E_k(n)$ when $n$ tends to infinity.
  \item[Step 7] Based on (Step 6), we further prove sharp double inequalities of $E_k(n)$ for as possible as smaller $n$.
\end{description}
We would like to make several remarks regarding the MC algorithm. Firstly, it should be worth noting that (Step 2) plays an important role in the MC algorithm. The idea of the above algorithm is originated from Mortici~\cite{Mor-1} and Gosper's algorithm~(see Chapter 5 of ~ Petkovsek, Wilf and  Zeilberger~\cite{PWZ}). Secondly, if one cannot find the general term formula for $\mathrm{MC}_k(m)$, then only the finite continued fraction approximation solution can be provided. In this case, it predicts that finding a continued fraction representation started  from this series is ``hopeless", one should try other series expressions.

Thirdly, we explain how to look for all the related coefficients in $\mathrm{MC}_k(m)$. To do that, the initial-correction function $\mathrm{MC}_0(m)$ is vital. Let $\deg \mathrm{MC}_0(m)=-\kappa_0\in \mathbb{Z}$, and denote its first coefficient by $\lambda_0\neq 0$. It is not difficult to obtain $\kappa_0$ and $\lambda_0$, which satisfy the following condition:
\begin{align}
\min_{\kappa,\lambda}\deg\left(\frac{\lambda}{m^{\kappa}}
-\frac{P_r(m)}{Q_s(m)}\frac{\lambda}{(m+1)^{\kappa}}-R(m)\right).
\end{align}
If $\kappa_0>0$, then $\mathrm{MC}_0(m)$ has the form $\frac{\lambda_0}{\Phi(\kappa_0; m)}$ with $\Phi(\kappa_0 ;m)$ monic polynomial of degree $\kappa_0$. Otherwise, it contains two terms. The first term is the form of a polynomial
of degree $-\kappa_0$ with the leading coefficient $\lambda_0$, while the second term has the form  $\frac{\lambda_0'}{\Phi(\kappa_0'; m)}$ just like the previous case. Next, we in turn find other coefficients in $\mathrm{MC}_0(m)$ by solving a series of linear equations.

Once one determines the initial-correction function $\mathrm{MC}_0(m)$, other correction functions $\mathrm{MC}_k(m)$ for $k\ge 1$ will become easy. Actually, one may employ two approaches to treat them. One of them is a power series expansion, another is putting the whole thing over a common denominator such that
\begin{align}
\deg\left(\mathrm{MC}_k(m)-\frac{P_r(m)}{Q_s(m)}
\mathrm{MC}_k(m+1)-R(m)\right)
\end{align}
is a strictly decreasing function of $k$.

Now, we explain how to do (Step 6). First by multiplying the formula \eqref{Second-DF} by $\frac{\prod_{i=1}^{uu}(a_im+c_i)!}
{\prod_{j=1}^{vv}(b_jm+d_i)!}
\frac{1}{q^m}$, then by adding these formulas from $m=n$ to $\infty$, finally by checking $$\lim_{n\rightarrow\infty}\frac{\prod_{i=1}^{uu}(a_in+c_i)!}
{\prod_{j=1}^{vv}(b_jn+d_i)!}
\frac{1}{q^n}\mathrm{MC}_k(n)=0,$$
in this way it is not difficult to get the desired results for the rate of convergence of the $k$-th correction error term $E_k(n)$.

Lastly, for the MC algorithm it should be stressed, however, that we only need a special solution of the difference equation $X_m-a(m)X_{m+1}=b(m)$. In other words, it is not necessary to find the general solution of the corresponding  homogeneous  difference equation $X_m-a(m)X_{m+1}=0$.

\begin{example} For the exponential function $\exp(z)=\sum_{m=0}^{\infty}\frac{z^m}{m!}$, we take $R(m)=1, q=1/z, uu=0$ and $vv=1$. So one needs to solve the following equation
\begin{equation}
X_m-\frac{z}{m+1}X_{m+1}=1.\label{Exponential function-DF-0}
\end{equation}
As $R(m)\equiv 1$, it is not hard to imagine that we should choose the initial-correction
$\mathrm{MC}_0(m)$ in the form $d+\frac{u_0}{m+v_0}$. With the help of \emph{Mathematica} software, we have
\begin{align*}
&\mathrm{MC}_0(m)-\frac{z}{m+1}\mathrm{MC}_0(m+1)-1\\
&=\frac{(d-1) m^3+
(-2 + 2 d + u_0 - 2 v_0 + 2 d v_0 - b z)m^2+
d_1m+c_1}{(1 + m) (m + v_0) (1 +m + v_0)},
\end{align*}
where $c_1$ is a constant and $d_1=(-1 + d + 2 u_0 - 3 v_0 + 3 d v_0 + u_0 v_0 - v_0^2 + d v_0^2 - b z -
   u_0 z - 2 b v_0 z)$. Now we enforce the first three coefficients in the numerator to be zero, and find that
$\mathrm{MC}_0(m)=1+\frac{z}{m+1-z}$ and $\deg \mathrm{T}_0(m)=3$. Similarly, in this way we may solve $\mathrm{MC}_1(m)$,$\mathrm{MC}_2(m)$, and $\mathrm{MC}_3(m)$ in turn as follows
\begin{align*}
\mathrm{MC}_1(m)=&1+\frac{z}{m+1-z+}\frac{z}{m+2-z},\\
\mathrm{MC}_2(m)=&1+\frac{z}{m+1-z+}\frac{z}{m+2-z+}\frac{2z}{m+3-z},\\
\mathrm{MC}_3(m)=&1+\frac{z}{m+1-z+}\frac{z}{m+2-z+}\frac{2z}{m+3-z+}\frac{3z}{m+4-z}.
\end{align*}
Further, we may check that  $\deg \mathrm{T}_1(m)=5$,  $\deg \mathrm{T}_2(m)=7$ and  $\deg \mathrm{T}_3(m)=9$. So one may guess that
$\mathrm{MC}_k(m)=1+\frac{z}{m+1-z+}\k_{j=1}^{k}\left(\frac{jz}{m+j+1-z}\right)$ and $\deg \mathrm{T}_k(m)=2k+3$. Then, by using \emph{Mathematica} software again, we can verify that these relations hold for $1\le k\le 20$. Hence these tests predict that our conjecture should be true. By a limiting process, finally we guess that $CF(m)=1+\frac{z}{m+1-z+}\k_{k=1}^{\infty}\left(\frac{kz}{m+k+1-z}\right)$ is a solution of \eqref{Exponential function-DF-0}, which will be proved in Sec.~10 by the Bauer-Muir transformation.
\end{example}

\begin{example} Suppose that $r\ge 0$ and $\Re x>\frac 12$. The infinite series $S(r):=\sum_{m=1}^{\infty}\frac{2m}{(m^2+r^2)^2}$ is called a \emph{Mathieu series}. Let the continued fraction $CF(r;x)$ with a parameter $r$ be defined by
\begin{align}
CF(r;x)=\frac{1}{\left(x-\frac 12\right)^2+\frac 14\left(1+4r^2\right)+}\K_{k=1}^{\infty}
\left(\frac{\kappa_k}{\left(x-\frac 12\right)^2+\lambda_k}\right),
\label{CF-definition}
\end{align}
where for $k\in\mathbb{N}$
\begin{align}
\kappa_k= -\frac{k^4\left(k^2 + 4 r^2\right)}{4 (2 k - 1) (2 k + 1)},\quad\lambda_k=\frac 14 (2 k^2 + 2 k + 1 + 4 r^2).
\end{align}
With the MC algorithm, we may guess that $CF(r;m)$ is a solution of the equation $X_m-X_{m+1}=\frac{2m}{\left(m^2+r^2\right)^2}$. For the details of its proof, the reader is referred to~\cite{CTZ-1}.
\end{example}
\begin{remark} The MC algorithm does not require factorization of polynomials, so it is easy to operate with computer program. It's also worth mentioning that the idea of the MC algorithm allows us to seek the continued fraction approximation solution for a more general class of the first order linear difference equation. The difference equation $X_m-X_{m+1}=\frac{1}{m+1}$ is an example of such class equation, its initial-correction can be solved as $\mathrm{MC}_0(m)=-\ln\left(m+1/2\right)-\frac{1/24}{(m+1/2)^2+7/40}$, which indeed may be used to approximate the harmonic sequence.
\end{remark}
On the one hand, in order to determine all related coefficients, we often use an appropriate symbolic computation software, which requires a huge amount of computations. On the other hand, the exact expression at each occurrence also takes a lot of space. Moreover, in order to guess the continued fraction formula, we have to do a lot of additional work. All theorems in Sec. 5 to 10 build on experimental results in this section. Hence, we focus on the rigorous proof of main conjectures, and omit some of the related details for guessing these formulas. Interested readers
may refer to Sec. 6 and 8 in preprint~\cite{CTZ-2}.

\section{Some preliminary lemmas for the theory of continued fractions}
\begin{lemma} The canonical contraction
of $b_0+\k\left(a_n/b_n\right)$ with
\begin{align*}
C_k=A_{2k},\quad D_k=B_{2k}\quad \mbox{for}\quad k\in\mathbb{N}_0
\end{align*}
exists if and only if $b_{2k}\neq 0$ for $k\in\mathbb{N}$, and is given by
\begin{align}
&b_0+
\begin{array}{ccccccc}
a_1 b_2& & a_2a_3b_4& & a_4a_5b_2b_6 &\\
\cline{1-1}\cline{3-3}\cline{5-5}\cline{7-7}
a_2+b_1b_2&-& a_3b_4+b_2(a_4+b_3b_4)&-&a_5b_6+b_4(a_6+b_5b_6)&
-\cdots
\end{array}\nonumber\\
&
\begin{array}{ccc}
& a_{2n}a_{2n+1}b_{2n-2}b_{2n+2}& \\
\cline{2-2}
-&a_{2n+1}b_{2n+2}+b_{2n}(a_{2n+2}+b_{2n+1}b_{2n+2})&-\cdots.
\end{array}\nonumber
\end{align}
\end{lemma}
\begin{proof} It follows from Theorem 12 and Eq.~(2.4.3) of L. Lorentzen, H. Waadeland~\cite[pp.~83\nobreakdash--84]{LW}. For applications, interested
readers  may refer to Berndt~\cite[p.~121, (14.2)]{Berndt} or \cite[p.~157]{Berndt}.\end{proof}

\begin{lemma}
$b_0+\k\left(a_n/b_n\right)\approx d_0+\k\left(c_n/d_n\right)$ if and only if there exists a sequence $\left\{r_n\right\}$ of complex numbers with
$r_0=1,r_n\neq 0$
for all $n\in \mathbb{N}$, such that
\begin{align}
d_0=b_0,\quad c_n=r_{n-1}r_n a_n,\quad d_n=r_n b_n\quad \mbox{for all
 $n\in \mathbb{N}$.}
\end{align}
\end{lemma}
\begin{proof} As to the proof, consult~\cite[p.~73, Theorem 9]{LW}.\end{proof}

One of the outstanding theorems in the theory of continued fractions is the result described by O. Perron~\cite{Perron} as the Bauer-Muir transformation; this theorem and limiting cases of it give rise to numerous and extremely interesting
consequences. For instance, see Andrews \emph{et al.}~\cite{ABJL}, Berndt \emph{et al.}~\cite{BLW}, Jacobsen~\cite{Jac-1}, Lorentzen and Waadeland~\cite[p.~77--82, Example 11 to 13]{LW}, and Perron~\cite[p.~25-37]{Perron}, etc. We shall use the Jacobsen's formulation in ~\cite{Jac-1}~(also see~\cite[p.~76, Theorem 11]{LW}).
\begin{lemma} (The Bauer-Muir transformation) The Bauer-Muir transform of $b_0+\k\big(a_m/b_m\big)$ with respect to
$\left\{r_m\right\}_{m=0}^{\infty}$ from $\mathbb{C}$ exists if and only if
\begin{align}
\phi_m=a_m-r_{m-1}\left(b_m+r_m\right)\neq 0,\quad m\ge 1.
\end{align}
If it exists, then it is given by
\begin{align*}
&b_0+\begin{array}{ccccccc}
a_1& & a_2& & a_3 &\\
\cline{1-1}\cline{3-3}\cline{5-5}\cline{7-7}
b_1&+&b_2&+&b_3&
+\cdots
\end{array}\\
=&b_0+r_0+
\begin{array}{ccccccc}
\varphi_1& & a_1 \phi_2/\phi_1& & a_2 \phi_3/\phi_2 &\\
\cline{1-1}\cline{3-3}\cline{5-5}\cline{7-7}
b_1+r_1&\!+\!&b_2+r_2-r_0\phi_2/\phi_1&\!+\!
&b_3+r_3-r_1\phi_3/\phi_2&
\!+\cdots.\nonumber
\end{array}
\end{align*}
\end{lemma}
For the continued fraction with all the elements positive, we may employ the following lemma to estimate the error term.
\begin{lemma} Let $\left\{a_k\right\}_{k=1}^{\infty}, \left\{b_k\right\}_{k=1}^{\infty}$ be two sequences of positive numbers. Let the error term $E_m$ be defined by $E_m:=\k_{k=1}^{\infty}\left(a_k/b_k\right)-\k_{k=1}^{m}\left(a_k/b_k\right)$. For all positive integer $m$, then
\begin{align}
\left|E_m\right|\le \frac{
\prod_{k=1}^{m+1}a_j}{B_{m+1}B_m},\label{Error-bounds}
\end{align}
where $A_n$ and $B_n$ denote the numerators and denominators of the $n$th approximant of $\k_{k=1}^{\infty}\left(a_k/b_k\right)$.
\end{lemma}
\begin{proof} As all the elements of the continued fraction $\k_{k=1}^{\infty}\left(a_k/b_k\right)$ are positive, we find easily  that the sequence $\left\{A_{2l-1}/B_{2l-1}\right\}_{l=1}^{\infty}$ is strictly decreasing, and $\left\{A_{2l}/B_{2l}\right\}_{l=1}^{\infty}$ is strictly increasing. Therefore,
\begin{equation}
\K_{k=1}^{2l}\left(\frac{a_k}{b_k}\right)<
\K_{k=1}^{\infty}\left(\frac{a_k}{b_k}\right)
<\K_{k=1}^{2l-1}\left(\frac{a_k}{b_k}\right).
\end{equation}
By means of the determinant formula of continued fractions~(for instance, see~\cite[p.~9, (1.2.10)]{LW}), we derive that
\begin{align}
\frac{A_{m+1}}{B_{m+1}}-\frac{A_m}{B_m}=\frac{(-1)^{m}
\prod_{k=1}^{m+1}a_k}{B_{m+1}B_m}.\label{adjacent-term}
\end{align}
Hence
\begin{align}
\left|E_m(x)\right|
<\left|\K_{k=1}^{m+1}\left(\frac{a_k}{b_k}\right)-\K_{k=1}^{m}
\left(\frac{a_k}{b_k}\right)\right|
=\left|\frac{A_{m+1}}{B_{m+1}}-\frac{A_m}{B_m}\right|=\frac{
\prod_{j=1}^{m+1}a_j}{B_{m+1}B_m}.
\end{align}
This finishes the proof of Lemma 4.4.\end{proof}
\begin{remark}
In the interest of saving space, we shall not give applications of Lemma 4.4, since the details is very similar to that in~\cite[p.~1044-1045]{CTZ-1}.
\end{remark}

\section{Catalan's constant}
Catalan's constant can be defined as
\begin{equation}
K=\sum_{m=0}^{\infty}\frac{(-1)^m}{(2m+1)^2}=0.9159655942\ldots,\label{Catalan's definition}
\end{equation}
which is one of those classical constants whose irrationality and transcendence remain unproven. 
Catalan's constant is also expressible as~(see \cite{MSW,CFormula})
\begin{equation}
K=\frac 12\sum_{m=0}^{\infty}\frac{4^m (m!)^2}{(2m)!(2m+1)^2}.\label{Catalan's definition-1}
\end{equation}
For further details of the history of Catalan's constants up to about 2013, the reader is referred to Bailey et al.~\cite{BBMW}.
\begin{theorem} Let $\Re x\ge 0$, and we write
\begin{equation}
CF(x):=\frac{1/2}{x+q_0+}\K_{k=1}^{\infty}\left(\frac{p_k}{x+q_k}\right),
\end{equation}
where
\begin{align}
p_k=\frac{(2k)^3(2k-1)^3}{4(4k-3)(4k-1)^2(4k+1)},\quad q_k=\frac{4k^2+2k-1}{2(4k-1)(4k+3)}.
\end{align}
Then for all non-negative integer $n$, we have
\begin{equation}
K=\frac 12\sum_{m=0}^{n-1}\frac{4^m (m!)^2}{(2m)!(2m+1)^2}+\frac 12\frac{4^n (n!)^2}{(2n)!}\cdot CF(n).
\end{equation}
In particular, $K=\frac 12CF(0)=\frac 12+CF(1)$.
\end{theorem}

\begin{lemma}For $k\ge 1$, let
$\kappa_{2k}=\frac{-(2k)^3}{2(4k-1)(4k+1)}, \kappa_{2k-1}=\frac{(2k-1)^3}{2(4k-3)(4k-1)}$,
$\lambda_{2k}=m$, and $\lambda_{2k-1}=1$.
Then the sequence $\{X_m\}_{m=0}^{\infty}=\left\{\frac{1/2}{m+}\k_{k=1}^{\infty}\left(\frac{\kappa_k}{\lambda_k}\right)\right\}_{m=0}^{\infty}$ satisfies the following  difference equation
\begin{equation}
X_{m}-\frac{2(m+1)}{2m+1}X_{m+1}=\frac{1}{(2m+1)^2}.\label{Catalan-DF}
\end{equation}
\end{lemma}
\begin{remark}
Lemma 5.2 confirms a question proposed in~\cite[p.~204]{CY}.
\end{remark}
\begin{proof} First we let
\begin{equation}
F(x)=b_0+\K_{k=1}^{\infty}\left(\frac{a_k}{b_k}\right),
\end{equation}
where
\begin{align*}
&a_{2k}=-\frac{(2k)^3}{2},\quad a_{2k-1}=\frac{(2k-1)^3}{2},\\
&b_{2k}=(4k+1)x,\quad b_{2k-1}=4k-1.
\end{align*}

Following Perron's approach in \cite[p.~29-32]{Perron}(also see Berndt~\cite[p.~159-161]{Berndt}), we now choose
\begin{equation}
r_{2k}=2k^2-(2x-1)k+2x^2+x+\frac 12,\quad r_{2k-1}=-2k+2x+1,
\end{equation}
and write
\begin{equation*}
\phi_{2k}=a_{2k}-r_{2k-1}\left(b_{2k}+r_{2k}\right),\quad
\phi_{2k-1}=a_{2k-1}-r_{2k-2}\left(b_{2k-1}+r_{2k-1}\right).
\end{equation*}
We easily show that $\phi_{2k}=\phi_{2k-1}=-\frac 12(2x+1)^3$. Moreover, it is not difficult to check that
$b_0+r_0=\frac 12(2x+1)^2$, $b_1+r_1=2(x+1)$, $b_{2k}+r_{2k}-r_{2k-2}=(4k-1)(x+1)$ and $ b_{2k+1}+r_{2k+1}-r_{2k-1}=4k+1$.
Then applying Lemma 4.3, we find that
\begin{equation*}
F(x)=b_0+r_0+\begin{array}{ccccccccc}
\phi_1& & a_1& & a_2 &&a_3&\\
\cline{1-1}\cline{3-3}\cline{5-5}\cline{7-7}\cline{9-9}
b_1+ r_1&\!+\!&b_2+r_2-r_0&\!+\!&b_3+r_3-r_1&+&b_4+r_4-r_2&\!+\!\cdots\!.
\end{array}
\end{equation*}
By means of Lemma 4.2, we get
\begin{align*}
&F(x)-\frac 12(2x+1)^2\\
=&\begin{array}{ccccccccccc}
-\frac 12(2x+1)^3& & a_1& & a_2 &&a_{2k-1}&&a_{2k}&\\
\cline{1-1}\cline{3-3}\cline{5-5}\cline{7-7}\cline{9-9}\cline{11-11}
2(x+1)&\!+\!&3(x+1)&\!+\!&5&\!+\!\cdots\!+\!&\!(4k-1)(x+1)\!&\!+\!&\!4k+1\!&\!+\!\cdots\!
\end{array}\\
=&\begin{array}{ccccccccccc}
-\frac 12(x+1)(2x+1)^3& & a_1& & a_2 &&a_{2k-1}&&a_{2k}&\\
\cline{1-1}\cline{3-3}\cline{5-5}\cline{7-7}\cline{9-9}\cline{11-11}
2(x+1)^2&\!+\!&3&\!+\!&\!5(x+1)\!&\!+\!\cdots\!+\!&\!4k-1\!&\!+\!&\!(4k+1)(x+1)\!&\!+\!\cdots\!
\end{array}\\
=&\frac{-\frac 12(x+1)(2x+1)^3}{(x+1)(2x+1)+F(x+1)}.
\end{align*}
From the above identity, one derives easily that
\begin{equation*}
F(x)\left(F(x+1)+(x+1)(2x+1)\right)=F(x+1)\frac 12 (2x+1)^2.
\end{equation*}
Dividing by $F(x)F(x+1)$ on both sides of the equation, we get
\begin{equation*}
1+\frac{(x+1)(2x+1)}{F(x+1)}=\frac{1}{F(x)}\frac 12 (2x+1)^2,~
\emph{i.e.}
~\frac{1/2}{F(x)}-\frac{2(x+1)}{2x+1}\frac{1/2}{F(x+1)}=\frac{1}{(2x+1)^2 }.
\end{equation*}
Finally, taking $X_m=\frac{1/2}{F(m)}$ and using Lemma 4.2 again, this will complete the proof of Lemma 5.2 at once.\end{proof}

\noindent{\emph{Proof of Theorem 5.1.} Firstly, multiplying  equation \eqref{Catalan-DF} by $\frac{4^m (m!)^2}{(2m)!}$, we find that}
\begin{equation}
\frac{4^m (m!)^2}{(2m)!}X_m-\frac{4^{m+1} \left((m+1)!\right)^2}{\left(2(m+1)\right)!}X_{m+1}=\frac{4^m (m!)^2}{(2m)!(2m+1)^2}.
\end{equation}
Then, by adding these formulas from $m=n$ to $m=\infty$, and noting $\lim_{m\rightarrow\infty}\frac{4^m (m!)^2}{(2m)!}X_m=0$, one deduces that
\begin{equation}
\frac{4^n (n!)^2}{(2n)!}\cdot X_n=\sum_{m=n}^{\infty}\frac{4^m (m!)^2}{(2m)!(2m+1)^2}.
\end{equation}
Finally, by Lemma 4.1, we see that $CF(m)$ is the even part of $X_m$, and this completes the proof of Theorem 5.1.\qed


\section{A companion of Catalan's constant}
The companion of Catalan's constant is defined as
\begin{equation}
G=\sum_{m=0}^{\infty}\frac{4^m (m!)^2}{(2m)!(4m+1)^2}.\label{Companion G-def}
\end{equation}
\begin{theorem}
For $\Re x\ge 0$ we define
\begin{equation}
CF(x):=\frac{1/8}{x+}\K_{k=1}^{\infty}\left(\frac{p_k}{x+q_k}\right),
\end{equation}
where
\begin{align*}
p_k=\begin{cases}
\frac{(2k-1)^4}{16(4k-3)(4k+1)}& \mbox{if $k$ is odd},\\
\frac{(2k)^4}{16(4k-3)(4k+1)}  &\mbox{if $k$ is even},
\end{cases}\quad
q_k=\begin{cases}
-\frac 12& \mbox{if $k$ is odd},\\
 0&\mbox{if $k$ is even}.
\end{cases}
\end{align*}
For all $n\in \mathbb{N}_0$, then
\begin{equation}
G=\sum_{m=0}^{n-1}\frac{4^m (m!)^2}{(2m)!(4m+1)^2}+\frac{4^n (n!)^2}{(2n)!}\cdot CF(n).
\end{equation}
\end{theorem}
\begin{lemma} We let $\kappa_{2k}=\frac{(4k)^4}{16(8k-3)(8k+1)}$, $\kappa_{2k-1}=\frac{(4k-3)^4}{16(8k-7)(8k-3)}$, $\lambda_{2k}=m$ and $\lambda_{2k-1}=m-\frac 12$. Then sequence $\{X_m\}_{m=0}^{\infty}=\left\{\frac{1/8}{m+}\k_{k=1}^{\infty}\left(\frac{\kappa_k}{\lambda_k}\right)\right\}_{m=0}^{\infty}$ is a special solution of the following equation
\begin{equation}
X_{m}-\frac{2(m+1)}{2m+1}X_{m+1}=\frac{1}{(4m+1)^2}.\label{Companion-G-DF}
\end{equation}
\end{lemma}
\begin{proof} First we write
\begin{equation}
F(x)=b_0+\K_{k=1}^{\infty}\left(\frac{a_k}{b_k}\right),
\end{equation}
where
\begin{align*}
&a_{2k}=\frac{(4k)^4}{16},\quad a_{2k-1}=\frac{(4k-3)^4}{16},\\
&b_{2k}=(8k+1)x,\quad b_{2k-1}=(8k-3)(x-1/2).
\end{align*}
We now take
\begin{equation}
r_{2k}=4k^2-(4x-1)k+2x^2+\frac 18,\quad r_{2k-1}=4k^2-(4x+1)k+2x^2+x+\frac 18,
\end{equation}
which are polynomials of degree two in $k$.  Then  define
\begin{align}
\phi_{2k}=a_{2k}-r_{2k-1}\left(b_{2k}+r_{2k}\right),~~\phi_{2k-1}=a_{2k-1}-r_{2k-2}\left(b_{2k-1}+r_{2k-1}\right),
\end{align}
hence $\phi_{2k}=\phi_{2k-1}=-\frac {(4x+1)^4}{64}$. Moreover, it is not difficult to verify that
\begin{align*}
&b_0+r_0=\frac 18(4x+1)^2,\quad b_1+r_1=\frac 18 ( 16 x^2+ 16 x +5),\\
&b_{2k}+r_{2k}-r_{2k-2}=( 8 k-3) (x+1),\quad b_{2k+1}+r_{2k+1}-r_{2k-1}=(8 k +1) (x +1/2).
\end{align*}
By means of Lemma 4.3, we find that
\begin{align*}
F(x)=&b_0+r_0+\begin{array}{ccccccccc}
\phi_1& & a_1& & a_2 &&a_3&\\
\cline{1-1}\cline{3-3}\cline{5-5}\cline{7-7}\cline{9-9}
b_1+ r_1&+&b_2+r_2-r_0&+&b_3+r_3-r_1&+&b_4+r_4-r_2&+\cdots.
\end{array}
\end{align*}
Applying Lemma 4.2 twice, it follows from the foregoing equality that
\begin{align*}
&F(x)-\frac 18(4x+1)^2\\
=&\begin{array}{ccccccccccc}
-\frac {(4x+1)^4}{64}& & a_1& & a_2 &&a_{2k-1}&&a_{2k}&\\
\cline{1-1}\cline{3-3}\cline{5-5}\cline{7-7}\cline{9-9}\cline{9-9}
\!\frac 18\!(\!16x^2\!+\!16x\!+\!5\!)\!&\!+\!&\!5\!(\!x\!+\!1\!)\!&\!+\!&\!9\!(\!x\!+\!\frac 12\!)\!&\!+\!\cdots\!+\!&\!(\!8k\!-\!3\!)\!(\!x\!+\!1\!)\!&\!+\!&(8k\!+\!1)\!(\!x\!+\!\frac 12\!)\!&\!+\!\cdots\!
\end{array}\\
=&
\begin{array}{ccccccccccc}
-\frac {(x+1)(4x+1)^4}{64}& & a_1& & a_2 &&a_{2k-1}&&a_{2k}&\\
\cline{1-1}\cline{3-3}\cline{5-5}\cline{7-7}\cline{9-9}\cline{9-9}
\!\frac 18\!(\!16x^2\!+\!16x\!+\!5\!)\!(\!x\!+\!1\!)\!&\!+\!&\!5\!&\!+\!&\!9\!(\!x\!+\!\frac 12\!)\!&\!+\!\cdots\!+\!&\!8k\!-\!3\!\!&\!+\!&(8k\!+\!1)\!(\!x\!+\!\frac 12\!)\!&\!+\!\cdots\!
\end{array}\\
=&\begin{array}{ccccccccccc}
-\frac {(4x+1)^4}{64(x+\frac 12)}& & a_1& & a_2 &&a_{2k-1}&&a_{2k}&\\
\cline{1-1}\cline{3-3}\cline{5-5}\cline{7-7}\cline{9-9}\cline{9-9}
\!\frac 18\!(\!16x^2\!+\!16x\!+\!5\!)\!\frac{\!x\!+\!1\!}{\!x\!+\!\frac 12\!}\!&\!+\!&\!5\!(\!x\!+\!\frac 12\!)\!&\!+\!&\!9\!(\!x\!+\!1\!)\!&\!+\!\cdots\!+\!&\!(\!8k\!-\!3\!)\!(\!x\!+\!\frac 12\!)\!&\!+\!&(8k\!+\!1)\!(\!x\!+\!1\!)\!&\!+\!\cdots\!
\end{array}\\
=&\frac{-\frac {(x+1)(4x+1)^4}{32(2x+1)}}{\frac{(x+1)(4x+1)^2}{4(2x+1)}+F(x+1)}.
\end{align*}
Therefore,
\begin{equation*}
F(x)\left(F(x+1)+\frac{(x+1)(4x+1)^2}{4(2x+1)}\right)=F(x+1)\frac 18(4x+1)^2.
\end{equation*}
Dividing by $F(x)F(x+1)$ on both sides of the equation, we obtain that
\begin{equation}
\frac{1/8}{F(x)}-\frac{2(x+1)}{(2x+1)}\frac{1/8}{F(x+1)}=\frac{1}{(4x+1)^2}.
\end{equation}
Finally taking $X_m=\frac{1/8}{F(m)}$ and employing Lemma 4.2 once more, we get the desired result.\end{proof}

\noindent{\emph{Proof of Theorem 6.1.} By multiplying both sides of \eqref{Companion-G-DF} by $\frac{4^m (m!)^2}{(2m)!}$ , we find that}
\begin{equation}
\frac{4^m (m!)^2}{(2m)!}X_m-\frac{4^{m+1} \left((m+1)!\right)^2}{\left(2(m+1)\right)!}X_{m+1}=\frac{4^m (m!)^2}{(2m)!(4m+1)^2}.
\end{equation}
Then let us sum this equality from $m=n$ to $\infty$, which will readily complete the proof of Theorem 6.1.\qed

\section{Antitrigonometric function arcsine}
For $0<w<4$, we define~(see Lorentzen and Waadeland~\cite[P.~297]{LW} and Wall~\cite[p.~345]{Wall})
\begin{equation}
g(w):=\sum_{m=0}^{\infty}\frac{w^m (m!)^2}{(2m)!(2m+1)}=\frac{4\arcsin\frac{\sqrt {w}}{2}}{\sqrt{w(4-w)}}.
\end{equation}
Note that $\left(g(1),g(2),g(3)\right)=\left(\frac{2 \pi}{3 \sqrt 3}, \frac{\pi}{2}, \frac{4\pi}{3\sqrt 3}\right).$

\begin{theorem}
Let $0<w<4$, $\Re x\ge 0$ and
\begin{equation}
CF(w;x):=\frac{2/(4-w)}{x+ q_0+}\K_{k=1}^{\infty}\left(\frac{p_k}{x+q_k}\right),
\end{equation}
where
\begin{align}
p_k=-\frac{2 w k(2k-1)}{(4-w)^2},\quad q_k=\frac{2+(4+w)k}{4-w}.
\end{align}
For $n\in \mathbb{N}_0$, one has
\begin{equation}
g(w)=\sum_{m=0}^{n-1}\frac{w^m (m!)^2}{(2m)!}+\frac{w^n (n!)^2}{(2n)!}\cdot CF(w;n).
\end{equation}
\end{theorem}
\begin{lemma}
Let sequences $\left\{p_k\right\}_{k=1}^{\infty}$ and $\left\{q_k\right\}_{k=0}^{\infty}$ is given as Theorem 7.1. If $0<w<4$, then $\{X_m\}_{m=0}^{\infty}=\left\{\frac{2/(4-w)}{m+ q_0+}\k_{k=1}^{\infty}\left(\frac{p_k}{m+q_k}\right)\right\}_{m=0}^{\infty}$ satisfies the following relation
\begin{equation}
X_{m}-\frac{w(m+1)}{2(2m+1)}X_{m+1}=\frac{1}{2m+1}.\label{arcsine-DF}
\end{equation}
\end{lemma}
\begin{proof} We let
\begin{equation}
F(x)=x+q_0+\K_{k=1}^{\infty}\left(\frac{p_k}{x+q_k}\right).
\end{equation}
We choose $u=-\frac{w}{4-w}$, $v=\frac{wx}{4-w}$. Let $r_k=uk+v$ and
\begin{equation}
\phi_k=p_k-r_{k-1}\left(q_k+r_k\right),\quad k\ge 1.
\end{equation}
It is easy to verify that $\phi_k$ is a constant,\emph{ i.e.} $\phi_k=-\frac{2w(x+1)(2x+1)}{(4-w)^2}$. By means of the Bauer-Muir transformation, we get
\begin{align*}
F(x)=&x+q_0+r_0+\begin{array}{ccccccc}
\phi_1& & p_1& & p_2 &\\
\cline{1-1}\cline{3-3}\cline{5-5}\cline{7-7}
x+q_1+ r_1&\!+\!&x+q_2+r_2-r_0&\!+\!&x+q_3+r_3-r_1&\!+\!\cdots\!
\end{array}\\
=&\frac{2(2x+1)}{4-w}+\begin{array}{ccccccc}
-\frac{2w(x+1)(2x+1)}{(4-w)^2}& & p_1& & p_k &\\
\cline{1-1}\cline{3-3}\cline{5-5}\cline{7-7}
\frac{2(3+2x)}{4-w}&+&x+1+q_1&\!+\!\cdots\!+\!&x+1+q_k&+\cdots
\end{array}\\
=&\frac{2(2x+1)}{4-w}+\frac{-\frac{2w(x+1)(2x+1)}{(4-w)^2}}{\frac{w(x+1)}{4-w}+F(x+1)}.
\end{align*}
It is not difficult to check that
\begin{equation}
\left(F(x)-\frac{2(2x+1)}{4-w}\right)F(x+1)+F(x)\frac{w(x+1)}{4-w}=0,
\end{equation}
This in turn implies that
\begin{equation}
F(x)\left(F(x+1)+\frac{w(x+1)}{4-w}\right)=\frac{2(2x+1)}{4-w}F(x+1).
\end{equation}
By taking reciprocals, one has
\begin{equation}
\frac{2}{4-w}\frac{1}{F(x)}-\frac{w(x+1)}{2(2x+1)}\frac{2}{4-w}\frac{1}{F(x+1)}=\frac{1}{2x+1}.
\end{equation}
Set $X_m=\frac{2}{4-w}\frac{1}{F(m)}$, Lemma 7.2 follows from the previous equality readily.\end{proof}
\begin{remark}
We may use another solution $(u,v)=\left(\frac{4}{w-4},\frac{2+4x}{w-4}\right)$ to prove the foregoing lemma. In this case, we have $
\phi_k=-\frac{2wx(2x-1)}{(4-w)^2}$.
\end{remark}
\noindent{\emph{Proof of Theorem 7.1.} Firstly, multiplying~\eqref{arcsine-DF} by $\frac{w^m (m!)^2}{(2m)!}$, and by next applying telescoping method to this equality, we complete the proof of Theorem 7.1.}\qed

\begin{remark}
Let $\kappa_1=\frac {2}{4-w}, \lambda_1=x$. For $k\ge 1$, we define $\kappa_{2k+1}=\frac{wk}{4-w}, \kappa_{2k}=\frac{2(2k-1)}{4-w}$, $\lambda_{2k+1}=x+1, \lambda_{2k}=1$. If $0<w<4$, $\frac{2/(4-w)}{x+ q_0+}\k_{k=1}^{\infty}\left(\frac{p_k}{x+q_k}\right)=\k_{k=1}^{\infty}\left(\frac{\kappa_k}{\lambda_k}\right)$ for $\Re x\ge 0$.
\end{remark}
\section{A special Lerch transcendent }
The Lerch transcendent~\cite{Lerch} is given by
\begin{equation}
\Phi (z,s,\alpha )=\sum _{m=0}^{\infty }\frac {z^{m}}{(m+\alpha )^{s}}.
\end{equation}
In this section, we shall give a continued fraction expression for $\Phi (z,1,\alpha )$, which usually appears in BBP-type formulas for  various mathematical constants. Stieltjes~\cite{Stieltjes} studied even its continued fraction expansion, also see Wall~\cite[p.~359-360]{Wall}.


\begin{theorem}
Let $\Re x\ge 0$, $\alpha>0$, $z\neq 1$, and continued fraction $CF(z,\alpha;x)$ be defined by
\begin{equation}
CF(z,\alpha;x):=\frac{1/(1-z)}{x+ q_0+}\K_{k=1}^{\infty}\left(\frac{p_k}{x+q_k}\right),
\end{equation}
where for $k\in\mathbb{N}_0$
\begin{align}
p_k=-\frac{z~k^2}{(1-z)^2},\quad q_k=\alpha+\frac{(1+z)k+z}{1-z}.
\end{align}
Then for all $n\in \mathbb{N}_0$, and $ |z|\le 1$ with $z\ne 1$, we have
\begin{equation}
\Phi (z,1,\alpha )=\sum _{m=0}^{n-1 }\frac {z^{m}}{m+\alpha }+z^{n}\cdot CF(z,\alpha;n).
\end{equation}
\end{theorem}

\begin{lemma}
Let sequences $\left\{p_k\right\}_{k=1}^{\infty}$ and $\left\{q_k\right\}_{k=0}^{\infty}$ be defined as Theorem 8.1, and
\begin{equation}
X_m=\frac{1/(1-z)}{m+ q_0+}\K_{k=1}^{\infty}\left(\frac{p_k}{m+q_k}\right).
\end{equation}
If $ |z|\le 1$ and $z\ne 1$, the sequence $\left\{X_m\right\}_{m=0}^{\infty}$ satisfies the following relation
\begin{equation}
X_{m}-z X_{m+1}=\frac{1}{m+\alpha}.\label{Lerch transcendent-DF}
\end{equation}
\end{lemma}
\begin{proof} We let $a_k=p_k$, $b_k=x+q_k$, and
\begin{equation}
F(x)=b_0+\K_{k=1}^{\infty}\left(\frac{a_k}{b_k}\right).
\end{equation}

We now choose $u=-\frac{1}{1-z}$, $v=-\frac{x+\alpha}{1-z}$. Let $r_k=uk+v$ and
\begin{equation}
\phi_k=a_k-r_{k-1}\left(b_k+r_k\right),\quad k\ge 1.
\end{equation}
Note that $\phi_k=-\frac{z(x+\alpha-1)^2}{(1-z)^2}$ and $b_{k+1}+r_{k+1}-r_{k-1}=b_k-1$. It follows from Lemma 4.3 that
\begin{align*}
F(x)=&b_0+r_0+\begin{array}{ccccccc}
\phi_1& & a_1& & a_k &\\
\cline{1-1}\cline{3-3}\cline{5-5}\cline{7-7}
b_1+ r_1&\!+\!&b_2+r_2-r_0&\!+\!\cdots\!+\!&b_{k+1}+r_{k+1}-r_{k-1}&\!+\!\cdots\!
\end{array}\\
=&\frac{(1 - x-\alpha)z}{1-z}+\begin{array}{ccccccc}
-\frac{z(x+\alpha-1)^2}{(1-z)^2}& & a_1& & a_k &\\
\cline{1-1}\cline{3-3}\cline{5-5}\cline{7-7}
\frac{(2 - x-\alpha)z}{1-z}&\!+\!&b_1-1&\!+\!\cdots\!+\!&b_k-1&+\cdots
\end{array}\\
=&\frac{(1 - x-\alpha)z}{1-z}+\frac{-\frac{z(x+\alpha-1)^2}{(1-z)^2}}{F(x-1)-\frac{x+\alpha-1}{1-z}}.
\end{align*}

Replacing $x$ by $x+1$, we find that
\begin{equation}
F(x)\left(F(x+1)+\frac{(x+\alpha)z}{1-z}\right)=F(x+1)\frac{x+\alpha}{1-z}.
\end{equation}
Dividing by $F(x)F(x+1)$ on both sides of the equation, one has
\begin{equation*}
\frac{x+\alpha}{1-z}\frac{1}{F(x)}=1+\frac{(x+\alpha)z}{1-z}\frac{1}{F(x+1)},
~\emph{i.e.}~
\frac{1}{1-z}\frac{1}{F(x)}-z\frac{1}{1-z}\frac{1}{F(x+1)}=\frac {1}{x+\alpha}.
\end{equation*}
Finally by taking $X_m=\frac{1}{1-z}\frac{1}{F(m)}$, we shall complete the proof of Lemma 8.2 at once.\end{proof}

\noindent{\emph{Proof of Theorem 8.1.} Now multiplying both sides in~\eqref{Lerch transcendent-DF} by $z^m $, we deduce that}
\begin{equation}
z^m X_{m}-z^{m+1} X_{m+1}=\frac{z^m}{m+\alpha},
\end{equation}
Theorem 8.1 follows from telescoping the foregoing equality readily.\qed

\begin{remark}
Let $\kappa_1=\frac{1}{1-z}, \lambda_1=x+\alpha,$ and for $k\ge 1$ $\kappa_{2k+1}=\frac{k}{1-z}, \kappa _{2k}=\frac{kz}{1-z}$, $\lambda_{2k+1}=x+\alpha, \lambda_{2k}=1$. If $|z|\le 1$ with $z\neq 1$, $\frac{1/(1-z)}{x+ q_0+}\k_{k=1}^{\infty}\left(\frac{p_k}{x+q_k}\right)=\k_{k=1}^{\infty}\left(\frac{\kappa_k}{\lambda_k}\right)$ for $\Re x\ge 0$.
\end{remark}

\section{A companion of Landau's constants}
With the help of \emph{Mathematica} software, we have
\begin{equation}
L=\sum_{m=0}^{\infty}\frac{(-1)^m(2m)!^2}{16^m(m!)^4}=\frac{\Gamma(1/4)}{2 \sqrt{\pi}~\Gamma(3/4)},
\end{equation}
which is closely connected with the well-known Landau's constants $G(n)=\sum_{m=0}^{n}\frac{(2m)!^2}{16^m(m!)^4}$. For further details of the history of Landau's constants up to about 2015, interested reader is referred to~\cite{Cao1}. Indeed,
the relation between Landau's constants and the companion of Landau's constants is very similar to that between the harmonic numbers $\sum_{m=1}^{n}\frac 1m$ and the series $\sum_{m=1}^{\infty} \frac{(-1)^{m-1}}{m}=\ln 2$.
\begin{theorem}
Let $\Re x\ge 0$ and continued fraction $CF(x)$ be defined by
\begin{equation}
CF(x):=\frac 12+\frac{1/4}{x+ 1/4+}\K_{k=1}^{\infty}\left(\frac{p_k}{x+1/4}\right),
\end{equation}
where for $k\ge 1$
\begin{align}
p_{2k}=\frac{k(2k+1)}{2},\quad p_{2k-1}=\frac{(4k-1)^2}{16}.
\end{align}
Then for all $n\in \mathbb{N}_0$, we have
\begin{equation}
L=\frac{\Gamma(1/4)}{2 \sqrt{\pi}~\Gamma(3/4)}=\sum_{m=0}^{n-1}\frac{(2m)!^2}{(-16)^m(m!)^4}+\frac{(2n)!^2}{(-16)^n(n!)^4}\cdot CF(n).
\end{equation}
\end{theorem}
\begin{lemma}
Let sequence $\left\{p_k\right\}_{k=1}^{\infty}$ be defined as Theorem 9.1, and
\begin{equation}
X_m=\frac 12+\frac{1/4}{m+ 1/4+}\K_{k=1}^{\infty}\left(\frac{p_k}{m+1/4}\right).
\end{equation}
Then the sequence $\left\{X_m\right\}_{m=0}^{\infty}$ is a special solution of the following equation
\begin{equation}
X_{m}+\frac{(2m+1)^2}{4(m+1)^2} X_{m+1}=1.\label{Landau's companion-DF}
\end{equation}
\end{lemma}
\begin{proof} For $k\ge 0$ let $b_k=x+1/4$, and define
\begin{equation}
F(x)=b_0+\K_{k=1}^{\infty}\left(\frac{p_k}{b_k}\right).
\end{equation}
Set
\begin{equation}
r_{2k}=k+\frac{5-8x^2}{4(4x+3)},\quad r_{2k-1}=k-\frac{(2x+1)^2}{2(4x+3)},
\end{equation}
and
\begin{equation*}
~~\phi_{2k}=p_{2k}-r_{2k-1}\left(b_{2k}+r_{2k}\right),~~\phi_{2k-1}=p_{2k-1}-r_{2k-2}\left(b_{2k-1}+r_{2k-1}\right),~~
\end{equation*}
thus $\phi_{2k}=\phi_{2k-1}=\frac {(2x+1)^2(x+1)^2}{(4x+3)^2}$. It follows from Lemma 4.3 that
\begin{align*}
F(x)=&b_0+r_0+\begin{array}{ccccccc}
\phi_1& & p_1& & p_2 &\\
\cline{1-1}\cline{3-3}\cline{5-5}\cline{7-7}
b_1+ r_1&+&b_2+r_2-r_0&+&b_3+r_3-r_1&+\cdots
\end{array}\\
=&x+\frac 14+\frac{5-8x^2}{4(4x+3)}+
\begin{array}{ccccccc}
\frac {(2x+1)^2(x+1)^2}{(4x+3)^2}& & p_1& & p_2 &\\
\cline{1-1}\cline{3-3}\cline{5-5}\cline{7-7}
b_1+1-\frac{(2x+1)^2}{2(4x+3)}&+&b_2+1&+&b_3+1&+\cdots
\end{array}\\
=&\frac{2(x+1)^2}{4x+3}+
\begin{array}{ccccccc}
\frac {(2x+1)^2(x+1)^2}{(4x+3)^2}& & p_1& & p_2 &\\
\cline{1-1}\cline{3-3}\cline{5-5}\cline{7-7}
-\frac{(2x+1)^2}{2(4x+3)}+b_0+1&+&b_1+1&+&b_2+1&+\cdots
\end{array}\\
=&\frac{2(x+1)^2}{4x+3}+\frac{\frac {(2x+1)^2(x+1)^2}{(4x+3)^2}}{-\frac{(2x+1)^2}{2(4x+3)}+F(x+1)}.
\end{align*}
Hence,
\begin{equation*}
F(x)\left(F(x+1)-\frac{(2x+1)^2}{2(4x+3)}\right)=F(x+1)\frac {2(x+1)^2}{4x+3}.
\end{equation*}
Dividing by $F(x)F(x+1)$ on both sides of the equation, we obtain that
\begin{equation}
\frac{1/4}{F(x)}+\frac{(2x+1)^2}{4(x+1)^2}\frac{1/4}{F(x+1)}=\frac{4x+3}{8(x+1)^2}.
\end{equation}
Finally taking $X_m=\frac 12+\frac{1/4}{F(m)}$, it is not difficult to check that $X_m$ satisfies~\eqref{Landau's companion-DF}, thus this finishes the proof of Lemma 9.2.\end{proof}

\noindent{\emph{Proof of Theorem 9.1.} Multiplying both sides in \eqref{Companion-G-DF} by $\frac{(2m)!^2}{(-16)^m(m!)^4}$ , we find that}
\begin{equation}
\frac{(2m)!^2}{(-16)^m(m!)^4}X_m-\frac{\left((2(m+1))!\right)^2}{(-16)^{m+1}\left((m+1)!\right)^4}X_{m+1}=\frac{(2m)!^2}{(-16)^m(m!)^4}.
\end{equation}
Now Theorem 9.1 follows from telescoping the foregoing equality at once.\qed

\section{The exponential function}
It is well-known that the exponential function $\exp(z)$ is given as
\begin{equation}
\exp(z)=\sum_{m=0}^{\infty}\frac{z^m }{m!},\quad z\in \mathbb{C}.
\end{equation}
The investigation of the exponential function has a long-standing and rich history. For the research history and continued fraction formulas of the exponential function, we refer the reader to~\cite{eFormula}.
\begin{theorem}
Let $z$ and $x$ be complex number, and write
\begin{equation}
CF(z;x):=1+\frac{z}{x+1-z+}\K_{k=1}^{\infty}\left(\frac{kz}{x+k+1-z}\right).
\end{equation}
Then for all $n\in \mathbb{N}_0$ and $z\in \mathbb{C}$, we have
\begin{equation}
\exp(z)=\sum_{m=0}^{n-1}\frac{z^m }{m!}+\frac{z^n }{n!}\cdot CF(z;n).\label{Exponential function-CF}
\end{equation}
\end{theorem}
\begin{lemma}
Let $m\in \mathbb{N}_0$, $z\in \mathbb{C}$, and write
\begin{equation}
X_m=1+\frac{z}{m+1-z+}\K_{k=1}^{\infty}\left(\frac{kz}{m+k+1-z}\right).
\end{equation}
Then the sequence $\left\{X_m\right\}_{m=0}^{\infty}$ is a special solution of the following equation
\begin{equation}
X_m-\frac{z}{m+1}X_{m+1}=1.\label{Exponential function-DF}
\end{equation}
\end{lemma}
\begin{proof} We first suppose that $z\neq 0$. Otherwise, the assertion is trivial. Now let $a_k=kz$, $b_k=x+k+1-z$, and
\begin{equation}
F(x)=x+1-z+\K_{k=1}^{\infty}\left(\frac{kz}{x+k+1-z}\right).
\end{equation}
For $k\in \mathbb{N}$, set $\phi_k=a_k-r_{k-1}(b_k+r_k)$. It follows from Lemma 4.3~(the Bauer-Muir transformation) with modifying factors $\left\{r_k\right\}_{k=0}^{\infty}=\left\{z\right\}_{k=0}^{\infty}$ that
\begin{align*}
F(x)=&b_0+r_0+\begin{array}{ccccccc}
\phi_1& & a_1& & a_k &\\
\cline{1-1}\cline{3-3}\cline{5-5}\cline{7-7}
b_1+ r_1&\!+\!&b_2+r_2-r_0&\!+\!\cdots\!+\!&b_{k+1}+r_{k+1}-r_{k-1}&\!+\!\cdots\!
\end{array}\\
=&x+1+\begin{array}{ccccccc}
-(x+1)z& & a_1& & a_k &\\
\cline{1-1}\cline{3-3}\cline{5-5}\cline{7-7}
x+2&\!+\!&x+3-z&\!+\!\cdots\!+\!&x+k+2-z&\!+\!\cdots\!
\end{array}\\
=&x+1+\frac{-(x+1)z}{z+F(x+1)},
\end{align*}
here we used $\phi_k=-(x+1)z$ for $k\in\mathbb{N}$. It is easy to see that $F(x)\left(z+F(x+1)\right)=(x+1)F(x+1)$. Dividing by $F(x)F(x+1)$ on both sides of the equation, we derive that
\begin{equation}
\frac{1}{F(x)}-\frac{z}{x+1}\frac{1}{F(x+1)}=\frac{1}{x+1}.
\end{equation}
On taking $CF(z;x)=1+\frac{z}{F(x)}$, then $\{X_m\}_{m=0}^{\infty}= \{CF(z;m)\}_{m=0}^{\infty}$ satisfies equation \eqref{Exponential function-DF}, and this finishes the proof of the desired assertion.\end{proof}

\noindent{\emph{Proof of Theorem 10.1.} By multiplying the equation~\eqref{Exponential function-DF} by $\frac{z^m }{m!}$, we find that}
\begin{align}
\frac{z^m }{m!}X_m-\frac{z^{m+1}}{(m+1)!}X_{m+1}=\frac{z^m }{m!}.
\end{align}
Now, by applying telescoping the foregoing equality from $m=n$ to $\infty$, this will complete the proof of Theorem 10.1 at once.\qed
\begin{remark}
The classical formula $\exp(z)=1+\frac{z}{1-z+}\k_{k=1}^{\infty}\left(\frac{kz}{k+1-z}\right)$~(see Cuyt \emph{et al.}~\cite[p.~194, (11.1.4)]{CPV} or Jones and Thron~\cite[p.~272]{JT}) is the subcase $n=0$ in  \eqref{Exponential function-CF}. 
\end{remark}

\section{Conclusions}
From those examples in Sec.~5-10, we conclude that, for a specific Ramanujan series, the MC algorithm provides a useful tool for discovering its continued fraction representation. In other words, for a given Ramanujan series, the MC algorithm could tell us at least that there does not exist a ``simple" continued fraction expression. We shall show more continued fraction formulas elsewhere, and some conjectures are also proposed for further study. At the same time, these examples predict that it is an effective approach for us to investigate the first order difference equation by the theory of continued fractions. So our method should help advance the approximation theory, the first order linear difference equation, the theory of continued fractions and the hypergeometric function, etc. Furthermore, these continued fraction formulas could probably be used to study the irrationality and transcendence of the involved series~(see Ap\'ery~\cite{Ap}, Borwein \emph{et al.}~\cite{BPSZ}, Waldschmidt~\cite{Walds}, etc.).
\bibliographystyle{amsplain}

\end{document}